\documentclass[11pt,letterpaper]{article}
\usepackage{amsmath, amssymb, amsthm}
\usepackage{mathrsfs}
\usepackage{cite}
\usepackage{hyperref}
\usepackage{xcolor}
\usepackage{enumitem}
\usepackage{tikz-cd}

\newtheorem{theorem}{Theorem}[section]
\newtheorem{lemma}[theorem]{Lemma}
\newtheorem{corollary}[theorem]{Corollary}
\newtheorem{proposition}[theorem]{Proposition}
\newtheorem{definition}[theorem]{Definition}
\newtheorem{remark}[theorem]{Remark}
\newtheorem{example}[theorem]{Example}
\newtheorem{setup}[theorem]{Setup}

\newcommand{\Cl}{\operatorname{Cl}}
\newcommand{\Po}{\operatorname{Po}}
\newcommand{\Gal}{\operatorname{Gal}}
\newcommand{\disc}{\operatorname{disc}}
\newcommand{\Norm}{\mathcal{N}}
\newcommand{\OO}{\mathcal{O}}
\newcommand{\UU}{U}
\newcommand{\coloneqq}{\mathrel{\mathop:}=}

\title{Computing $p$-Class Group Structure in Real Quadratic Fields: A New Approach}
\author{Farahnaz Amiri}
\date{}

\begin{document}

\maketitle

\begin{abstract}
This article is the first in a series devoted to computing the class groups of real quadratic fields. We present a new relation between the class number and the index of unit groups. This relation generalizes Hilbert class field theory for real quadratic fields and establishes a bridge between class field theory, composition laws of binary forms of degree $p^n$, and ideal classes of order $p^n$, where $p$ is prime and $n$ is an arbitrary positive integer.




In this new framework, the well-known Hilbert class field becomes just one special example in a much larger family (see Remark \ref{Rema:Hilbert-identity}). While classical class field theory gives important theoretical insights, it is often not practical for determining the exact order of an ideal.  To solve this problem, we develop new composition laws that work for ideals of any order. This goes far beyond the previously known cases of order 2 and 3—established by Gauss and Bhargava, respectively. This gives us a powerful and concrete method to study the class group, leading to several new results.

\end{abstract}


\section{Introduction}

In this paper, we propose a new method for studying the class groups of real quadratic fields. To contextualize our work, we begin by recalling three foundational approaches: composition laws, Hilbert class field theory, and the analytic class number formula. We will outline these methods and discuss how they can be further developed.

\subsection{Historical Context and Higher Composition Laws}
So far, the theory linking composition laws to a correspondence between polynomials of degree $p$ and ideals of order $p$ has only been developed for the primes $2$ and $3$. We begin by briefly discussing these foundational cases before presenting our results for an arbitrary prime $p$.

In fact, we study ideals of order $p^n$ in the class groups of real quadratic fields, for any prime $p$ and integer $n$. Previous research was largely confined to the $2$- and $3$-parts, with key contributions from Gauss \cite{Gauss1863}, Davenport \cite{davenportheilbronn1971}, and Bhargava \cite{BhargavaVarma2016}. Our work generalizes Gauss's composition laws to all odd primes $p$, and specifically to the $p^n$-torsion components of the class group.
\subsubsection{Genus Theory and the 2-part}

A cornerstone of this subject is the classical connection between binary quadratic forms and ideal classes in quadratic orders, formally captured by the following theorem.

\begin{theorem}[\cite{Cohen2013}, Theorems 5.2.8 and 5.2.9]
Let $D$ be a non-square integer.
\begin{itemize}
\item If $D < 0$, there exists a bijection between the $\operatorname{SL}_2(\mathbb{Z})$-equivalence classes of positive definite, primitive, integral binary quadratic forms of discriminant $D$ and the ideal class group of the quadratic order of discriminant $D$.

\item If $D > 0$, there exists a bijection between the $\operatorname{SL}_2(\mathbb{Z})$-equivalence classes of primitive, integral binary quadratic forms of discriminant $D$ and the \emph{narrow} ideal class group of the quadratic order of discriminant $D$.
\end{itemize}
\end{theorem}

In modern language, this correspondence has been generalized, as seen in the following result of Wood.

\begin{theorem}[Wood \cite{Wood2011}]
There exists a discriminant-preserving bijection:
\begin{footnotesize}
\[
\left\{\begin{array}{c}\text{twisted } \operatorname{GL}_2(\mathbb{Z}) \text{ equivalence classes}\\
\text{of non-degenerate binary}\\
\text{quadratic forms}\end{array}\right\}
\longleftrightarrow
\left\{\begin{array}{c}\text{isomorphism classes of } (C,I), \text{ where}\\
C \text{ is a non-degenerate oriented quadratic}\\
\text{ring, and } I \text{ is a full ideal class of } C\end{array}\right\}.
\]
\end{footnotesize}
That is, if $f \leftrightarrow (C,I)$, then $\operatorname{disc} f = \operatorname{disc} C$.
\end{theorem}
Gauss's pioneering work on quadratic fields \cite{Gauss1863} led him to conjecture that there are infinitely many real quadratic fields with class number one. This conjecture remains open. A more recent, probabilistic approach was proposed by Cohen-Lenstra \cite{CohenLenstra2006}. Their heuristics predict the distribution of the $q$-parts of the class groups of real quadratic fields for integers $q$. Let $p(q)$ denote the probability that the $q$-part is non-trivial. Their model predicts, for example:

\[
p(3) \approx 12.574\%,\quad p(5) \approx 3.772\%,\quad p(7) \approx 1.796\%,\quad p(9) \approx 1.572\%.
\]
Regarding the odd part, only Bhargava and Davenport have to date provided a description of this conjecture for the 3-part of the class group. We recall it here.

\subsubsection{Bhargava's Higher Composition Laws and the 3-part}

Manjul Bhargava's higher composition laws furnish an effective framework for examining the 3-torsion part \cite{Bhargava2004}.  
Let \(V_{\mathbb{R}}\) be the four-dimensional real vector space consisting of binary cubic forms  
\(ax^{3}+bx^{2}y+cxy^{2}+dy^{3}\) with coefficients \(a,b,c,d\in\mathbb{R}\), and let \(V_{\mathbb{Z}}\) be the lattice of forms having integer coefficients.  
The group \(\operatorname{GL}_{2}(\mathbb{Z})\) acts on \(V_{\mathbb{R}}\) through the \emph{twisted action}: for \(\gamma\in\operatorname{GL}_{2}(\mathbb{Z})\) and a form \(f(x,y)\),

\[
(\gamma f)(x,y)=\frac{1}{\det\gamma}\,f\bigl((x,y)\cdot\gamma\bigr).
\]

This action leaves the lattice \(V_{\mathbb{Z}}\) invariant.  
Of particular interest is the sublattice \(V_{\mathbb{Z}}^{*}\) of \emph{integer-matrix} forms, i.e. forms that can be written as  
\(f(x,y)=ax^{3}+3bx^{2}y+3cxy^{2}+dy^{3}\) with \(a,b,c,d\in\mathbb{Z}\).  
The lattice \(V_{\mathbb{Z}}^{*}\) has index \(9\) in \(V_{\mathbb{Z}}\) and is likewise preserved by \(\operatorname{GL}_{2}(\mathbb{Z})\).  
On \(V_{\mathbb{Z}}^{*}\) we introduce the \emph{reduced discriminant}

\[
\operatorname{disc}(f)=-\frac{1}{27}\operatorname{Disc}(f)=-3b^{2}c^{2}+4ac^{3}+4b^{3}d-abcd .
\]

Bhargava’s central result is the following bijection.

\begin{theorem}[\cite{Bhargava2004}, Theorem 13]
There is a natural bijection between the set of nondegenerate $\operatorname{SL}_2(\mathbb{Z})$ orbits on the space $V_{\mathbb{Z}}^*$ of integer-matrix binary cubic forms and the set of equivalence classes of triples $(\OO, I, \delta)$, where:
\begin{itemize}
\item $\OO$ is a nondegenerate oriented quadratic ring over $\mathbb{Z}$,
\item $I$ is an ideal of $\OO$,
\item $\delta$ is an invertible element of $\OO \otimes \mathbb{Q}$ such that $I^3 \subseteq \delta \cdot \OO$ and $N(I)^3 = N(\delta)$.
\end{itemize}
Two triples $(\OO, I, \delta)$ and $(\OO', I', \delta')$ are equivalent if there is an isomorphism $\phi: \OO \to \OO'$ and an element $\kappa \in \OO' \otimes \mathbb{Q}$ such that $I' = \kappa \phi(I)$ and $\delta' = \kappa^3 \phi(\delta)$. Under this bijection, the reduced discriminant of a binary cubic form is equal to the discriminant of the corresponding quadratic ring.
\end{theorem}

This geometric approach yielded profound arithmetic consequences. For instance, it enabled the precise computation of the average size of the 3-torsion in class groups.

\begin{theorem}[\cite{BhargavaVarma2016}, Theorem 2]
When orders in quadratic fields are ordered by their absolute discriminants:
\begin{itemize}
\item The average number of $3$-torsion elements in the class groups of imaginary quadratic orders is $1 + \frac{\zeta(2)}{\zeta(3)}$.
\item The average number of $3$-torsion elements in the class groups of real quadratic orders is $1 + \frac{1}{3} \cdot \frac{\zeta(2)}{\zeta(3)}$.
\end{itemize}
Note that $\frac{\zeta(2)}{\zeta(3)} \approx 1.36843 > 1$.
\end{theorem}

\subsubsection{Generalizations and Current Limitations}

Melanie Wood's Ph.D. thesis \cite{Wood2011} generalized composition laws to arbitrary base rings, establishing a relation between these laws and class groups of fields above quadratic fields, rather than directly with the class group of a quadratic field \cite[Theorem 1.4]{Wood2011}.

Existing methods for studying higher torsion, such as those for the 5-part, often rely on computationally complex generalizations. For instance, Bhargava's approach for the 5-part involves the analysis of sextic polynomials \cite[Section 2]{Bhargava2004}. Our method provides a more systematic alternative through the introduction of:

\paragraph{$(p,n)$-Higher Composition Laws.} This framework formalizes the relationship between ideal classes of order $p^n$ in $\Cl(N)$ and combinations of polynomial structures of degree $p^n$. This provides a unified extension of classical composition laws to higher exponents.

\subsection{Our Contribution: Generalized Composition Laws for Arbitrary Primes}
We introduce a generalized framework for composition laws that applies to every ideal class of order $p^n$ in the class group of a real quadratic field, for any prime $p$ and integer $n$. This framework has the potential to extend the classical results of Gauss on the $2$-part and the modern breakthroughs of Bhargava on the $3$-part of the class group of real quadratic fields.

A key distinction lies in the scope of the correspondence: whereas Gauss and Bhargava establish bijections for the entire $2$-part or $3$-part, our method guarantees, for every odd prime $p$ and integer $n$, a correspondence between ideals of order $p^n$ in the class group of real quadratic fields and degree $p^n$ polynomials.

The rationale for our approach is provided by the following technical theorem, which connects ideals of order $p^n$ in the class group to the properties of certain cyclic extensions.

\begin{setup}\label{setup:tower}
Let $p$ be an odd prime and $N$ a real quadratic field. Let $L/N$ be abelian with $G=\Gal(L/N)\cong(\mathbb{Z}/p^2)^2$, so $[L:N]=p^4$. Put $H\coloneqq L^{pG}$ (so $\Gal(H/N)\cong G/pG\cong(\mathbb{Z}/p)^2$ and $\Gal(L/H)=pG\cong(\mathbb{Z}/p)^2$), and let $M$ be a subfield with $N\subset M\subset H$, $[M:N]=p$. Then $\exp(G)=p^2=[H:N]$.
\end{setup}

\begin{theorem}[Main technical theorem]\label{thm:main-technical}
Let $N$ be a real quadratic field, $p$ an odd prime, and $[I]\in\Cl(N)$ the class of an ideal $I$ whose prime factors $\mathfrak{p}_i$ are unramified in $L/N$, where $N\subset M\subset H\subset L$ is as in Setup~\ref{setup:tower}. Suppose each $\mathfrak{p}_i$ has Frobenius $\sigma_i\in G\setminus \Gal(L/M)$, i.e.\ is inert in $M/N$. Then:
\begin{enumerate}
\item each $\mathfrak{p}_i$ has residue degree $p$ in $H/N$, hence splits in $H/M$;
\item $\mathfrak{p}_i$ is not totally split in $L/M$ (decomposition group $\langle\sigma_i^p\rangle$ of order $p$), and is inert in a degree-$p$ extension $K_i/H$ inside $L/H$, where $K_i=L^{\Lambda_i}$ for any $\Lambda_i\subset pG$ of order $p$ with $\sigma_i^p\notin\Lambda_i$;
\item the restricted transfer $\widetilde{\operatorname{Ver}}:G/pG\to pG$ vanishes, since $\operatorname{Ver}(g)=g^{[G:pG]}=g^{p^2}=1$; consequently the Artin symbol of $I\mathcal{O}_H$ in $\Gal(L/H)$ is trivial, i.e.\ $I\mathcal{O}_H\in\mathcal{P}_H^{\mathfrak{m}}\Norm_{L/H}(\mathcal{I}_L^{\mathfrak{m}})$, although no prime of $H$ above $\mathfrak{p}_i$ splits completely in $L/H$.
\end{enumerate}
\end{theorem}

\begin{proof}
(1) $f(H/N)=\operatorname{ord}(\sigma_iG')$ with $\sigma_i\notin pG$ gives $p$. (2) The decomposition group in $L/N$ is $\langle\sigma_i\rangle$ of order $p^2$ and in $L/M$ it is $\langle\sigma_i\rangle\cap\Gal(L/M)=\langle\sigma_i^p\rangle$, of order $p$ (as $\sigma_i\notin\Gal(L/M)$ and $\sigma_i^p\neq1$). (3) The primes of $H$ over $\mathfrak{p}_i$ are $p$ in number, each with Frobenius $\sigma_i^p$ in $\Gal(L/H)$; the product is $\sigma_i^{p^2}=1$, which is $\operatorname{Ver}(\sigma_i)$ by the Galois dictionary.
\end{proof}

Theorem \ref{thm:main-technical} provides the foundation for our composition laws. The connection arises because the norm map $\operatorname{Norm}_{M/N}$ from the unit group $\mathcal{O}^{\times}_M$ of the degree $p^n$ extension can be described by polynomials of degree $p^n$. Since $\operatorname{Norm}_{M/N}(\mathcal{O}^{\times}_M)$ forms a group, the product of elements represented by such polynomials can itself be represented by another polynomial of degree $p^n$. This multiplicative closure is the essence of a composition law.



\subsection{An Illustrative Example for Prime Degree}\label{sub-An-Illustrative-Example-for-Prime-Degree}
We now explain how Theorem~\ref{thm:main-technical} establishes a connection between composition laws for polynomials and the multiplication of ideal classes.  
The construction described below for ideal classes of order~$p$ generalizes naturally to ideal classes of order~$p^n$ for any integer $n \geq 1$ and any odd prime~$p$.

Let $N$ be a real quadratic field, and let $I_1$ and $I_2$ be ideals of $N$ whose ideal classes $[I_1]$ and $[I_2]$ have order~$p$, where $p$ is an odd prime, as characterized by Theorem~\ref{thm:main-technical}.  
Then there exist (non-unique) cyclic extensions $M_1/N$ and $M_2/N$ of degree~$p$ such that
\[
ord([I_i]) \;=\;
\bigl[\mathcal{O}_N^{\times} : \Norm_{M_i/N}(\mathcal{O}_{M_i}^{\times})\bigr],
\qquad i=1,2.
\]

Suppose that the product ideal $I_1 I_2$ also represents an ideal class of order~$p$.  
Then there exists a cyclic extension $M_3/N$ of degree~$p$ such that
\[
ord([I_1 I_2]) \; = \;
\bigl[\mathcal{O}_N^{\times} : \Norm_{M_3/N}(\mathcal{O}_{M_3}^{\times})\bigr].
\]

In extension $M_1/N$, the index of the norm group in the unit group is $p$:
\[
[\mathcal{O}^{\times}_N : \operatorname{Norm}_{M_1/N}(\mathcal{O}^{\times}_{M_1})] = p.
\]
Since $ \mathcal{O}^{\times}_N = \langle \varepsilon \rangle \times \{\pm 1\} $, and $p$ is odd it follows that
\[
-\varepsilon^p \in \operatorname{Norm}_{M_1/N}(\mathcal{O}^{\times}_{M_1}).
\]
This norm group corresponds to a minimal polynomial of degree $p$:
\[
x^p + a_{p-1}x^{p-1} + \dots + a_1 x + \varepsilon^p = 0.
\]
We define the following polynomial, obtained by removing the constant term from the norm equation:
\[
P(x) = x^p + a_{p-1}x^{p-1} + \dots + a_1 x.
\]

There are many such polynomials; we can choose one of them and say that it corresponds to the ideal $I_1$.

Similarly, for the ideal $I_2$ and its corresponding extension $M_2/N$, we have $-\varepsilon^p \in \operatorname{Norm}_{M_2/N}(\mathcal{O}^{\times}_{M_2}) $, yielding another polynomial:
\[
Q(x) = x^p + b_{p-1}x^{p-1} + \dots + b_1 x.
\]

Now consider the product $P(x)Q(x)$. When evaluated at elements of the norm group, this product yields the constant term $-\varepsilon^p \cdot (-\varepsilon^p) = \varepsilon^{2p}$.

For the product ideal $I_3 = I_1 I_2$ and its corresponding extension $M_3/N$, the unit $\varepsilon^p$ is contained in the norm group: $\varepsilon^p \in \operatorname{Norm}_{M_3/N}(\mathcal{O}^{\times}_{M_3})$. Consequently, $\varepsilon^{p} \in \operatorname{Norm}_{M_3/N}(\mathcal{O}^{\times}_{M_3})$ as well. This norm group also corresponds to a minimal polynomial of degree~$p$:
\[
W(x) =x^p + c_{p-1}x^{p-1} + \dots + c_1 x = -\varepsilon^{p}.
\]

Thus, both \(P(x)Q(x)\) and \(W(x)^2\), when evaluated on their respective norm groups, yield the same constant term \(\varepsilon^{2p}\). This shows that the product of the two polynomials \(P\) and \(Q\)---which correspond to the ideal classes of \(I_1\) and \(I_2\)---is represented by a third polynomial \(W\) of degree \(p\), which in turn corresponds to the product ideal class \(I_3\). In this way, the composition of polynomials mirrors the multiplication of ideal classes.

Note that this group operation differs from Gauss's original composition law. In our formulation, the product satisfies $P(x) Q(y) = W(z)^2$ instead of $P(x) Q(x) = W(x)$. Nevertheless, because these polynomials correspond to ideals, they emulate the group law of ideal classes, and we obtain a well-defined group action among these polynomials.

\begin{definition}
As explained above, for each extension of degree $p^n$ we can choose one polynomial of degree $p^n$ whose constant term equals $\varepsilon^{p^n}$, where $-\varepsilon$ is the fundamental unit of the real quadratic field $N$, and $p$ is an odd prime. We denote this family of polynomials by $\mathcal{F}$.
\end{definition}



\begin{theorem}\label{thm:main-bijection}
Let $p$ be an odd prime and let $n \geq 1$. 
There exists a natural bijection between binary forms of degree~$p^n$ in the family~$\mathcal{F}$ and ideal classes of order~$p^n$ in $\operatorname{Cl}(N)$.
\end{theorem}

Equivalently, the correspondence can be written as:
\[
\left\{\,
\begin{array}{c}
\text{binary forms of degree } p^n \\
\text{in the family } \mathcal{F}
\end{array}
\right\}
\quad\longleftrightarrow\quad
\left\{\,
\begin{array}{c}
\text{ideal classes } [I] \in \operatorname{Cl}(N) \\
\text{of order exactly } p^n
\end{array}
\right\}.
\]
This bijection provides a direct bridge between the geometry of forms and the arithmetic of ideal classes. Consequently, it not only generalizes the classical composition laws introduced by Gauss but also opens new avenues for computing and understanding higher torsion in class groups.


\begin{remark}
While our method relies on norm computations involving the fundamental unit of \( N \)---making it appear less direct than the universal composition laws of Gauss and Bhargava, which are governed by polynomials of degrees 2 and 3---it offers greater flexibility. 

Theorem \ref{thm:main-bijection} allows us to choose, for ideals of order~$p$ in $\operatorname{Cl}(N)$, extensions $M/N$ of degree~$p$ from among various possibilities $M_i$ (not just a single fixed extension). Consequently, we can select an $M_i$ for which the corresponding composition law is particularly simple to compute. Moreover, the same $M_i$ often works for many different quadratic fields $N$.

This flexibility creates the possibility of performing uniform and efficient computations across families of real quadratic fields.
\end{remark}

\subsection{Hilbert Class Field Theory and a New Practical Approach}

The second major approach to studying class groups is Hilbert class field theory. In this section, we recall its classical statement and then present a novel, more practical method for investigating ideals of order $p^n$ in the class group of a real quadratic field $N$.

While the classical Hilbert class field $H$ of a number field $N$ is the maximal unramified abelian extension, and its degree over $N$ equals the class number $h(N)$, its explicit construction is often computationally challenging.

We propose an alternative method that, for practical purposes, utilizes \emph{ramified} extensions $M/N$. This approach offers significant advantages:
\begin{itemize}
\item The discriminant conditions required (e.g., inert primes or specific factorization of minimal polynomials) are often simpler to verify computationally.
\item The behavior of the norm map $\Norm_{M/N}$ becomes more tractable and directly informative in these settings.
\end{itemize}

The following Corollary provides the simple usage for our method to detect ideal classes of order $p$ in $\Cl(N)$ in Theorem \ref{thm:main-bijection}, which for each odd $p$ can be explained as follows.

\begin{corollary}\label{cor:p-divisibility}
Let $N$ be a real quadratic field and let $p$ be an odd prime.  
Then the following statements are equivalent:
\begin{enumerate}[label=(\alph*)]
    \item The class number $|\Cl(N)|$ is divisible by $p$.
    \item For some cyclic extensions $M/N$ of degree $p$ such that the norm map on units
    \[
    \Norm_{M/N}\colon \mathcal{O}_M^{\times} \longrightarrow \mathcal{O}_N^{\times}
    \]
    is not surjective.
\end{enumerate}

Where the cyclic extension $M/N$ in~(b) is chosen to satisfy the following properties:
\begin{itemize}
    \item the absolute norm $\Norm_{N/\mathbb{Q}}(\mathfrak{d}(M/N))$ of the relative discriminant $\mathfrak{d}(M/N)$ is inert in $N$;
    \item there exists a cyclic extension $L/N$ of degree $p^k$ for some $k>1$ such that
    \[
    N \subset M \subset L;
    \]
    \item an ideal class $\mathfrak{c}\in \Cl(N)$ of order $p$ is represented by a prime ideal that is inert in the extension $M/N$.
\end{itemize}
\end{corollary}

The practical advantage of this method over direct computation in the Hilbert class field is demonstrated in the following example.

\begin{example}\label{ex:Qsqrt79}
Consider the field $N = \mathbb{Q}(\sqrt{79})$, which has class number $3$. Its Hilbert class field $H$ is a compositum of a non-Galois cubic extension, defined, for instance, by the polynomial $x^3 - x^2 - 4x + 2$, with $N$. Constructing and working with $H$ directly is complex.

Instead, we apply our method. Let $K/\mathbb{Q}$ be the cyclic cubic extension defined by $x^3 - 10x^2 + 21x - 11$, whose discriminant is inert in $N$. The compositum $M = KN$ forms a cyclic degree-$6$ extension of $\mathbb{Q}$, ramified at $37$. A key computation (see the program in Appendix~A) shows that the fundamental unit of $\mathbb{Q}(\sqrt{79})$ is not a norm from $M$. By Corollary~\ref{cor:p-divisibility}, this non-surjectivity of the norm map on units directly demonstrates that $3 \mid h(\mathbb{Q}(\sqrt{79}))$, offering a more computationally accessible path to the same conclusion.
\end{example}

\subsection{The Analytic Class Number Formula and a Deterministic Alternative}

A third classical approach to the study of class groups is provided by the analytic class number formula. This perspective is especially effective when the real quadratic field $N$ is realized as a subfield of a cyclotomic field. In this setting, the analytic class number formula yields conditional upper bounds on the size of $Cl(N)$.

However, results in analytic number theory often rely on heuristics, the Generalized Riemann Hypothesis (GRH), or smoothness assumptions to provide probabilistic or conditional guarantees. For example, a common line of inquiry uses the norm index to establish an upper bound for the class number of real quadratic fields \cite[Chapters 4 \& 8]{Washington2012}. While powerful, these results are inherently conditional and do not precisely determine the group structure.

Our work provides a distinct and complementary perspective. Whereas the analytic approach often yields conditional bounds, our method describes the $p^n$-order ideal in the class group of real quadratic fields \emph{accurately}—yielding its precise structure, not merely an upper bound. Furthermore, it achieves this unconditionally, without relying on the truth of the Riemann Hypothesis or other unproven conjectures.

\subsection{Strategy: From Pólya Groups to a Simplified Class Number Formula}

We begin by recalling the definition of the relative Pólya group associated to an extension $M/N$ and its relationship with the class group $\Cl(N)$.  
This relationship is encoded in the following exact sequence (see Lemma~\ref{lem:relative-polya-exact}):
\begin{multline*}\label{eq:main-exact-sequence} 1 \longrightarrow \operatorname{Ker}(\epsilon_{M/N}) \longrightarrow H^1(\Gal(M/N), \UU_M) \\ \longrightarrow \bigoplus_{\mathfrak{P}} \mathbb{Z}/e_{\mathfrak{P}}\mathbb{Z} \longrightarrow \frac{\operatorname{Po}(M/N)}{\epsilon_{M/N}(\Cl(N))} \longrightarrow 1. \end{multline*}
Here $\epsilon_{F/K}\colon \Cl(K)\to\Cl(F)$ denotes the transfer of ideal classes induced by the extension-of-ideals map
\[
j_{F/K}\colon I(K)\longrightarrow I(F), \qquad
\mathfrak{a}\longmapsto \mathfrak{a}\mathcal{O}_F .
\]

Our objective is to simplify the general formula relating the class number of $N$ to these invariants:
\begin{equation}\label{eq:general-formula}
h(N)\cdot \prod_i e_i(M/N)
=
\#\Po_{\mathrm{ram}}(M/N)\cdot
\#\Po_{\mathrm{un}}(M/N)\cdot
\# H^1\!\left(\Gal(M/N), \UU_M\right).
\end{equation}

Here \[
\operatorname{Po}(M/N) \cong \operatorname{Po}(M/N)_{\text{ram}} \oplus \operatorname{Po}(M/N)_{\text{un}}.
\]

Our strategy to achieve this simplification proceeds in three steps:

First, we establish several capitulation lemmas for the extension $M/N$ under specific ramification conditions. These results will allow us to control and ultimately eliminate the unramified part $\operatorname{Po}(M/N)_{\text{un}}$ from the formula.

Second, by working with a suitable quotient of the term $\#\operatorname{Po}_{\text{ram}}(M/N)$ and imposing the condition that $\operatorname{Norm}_{M/N}(\delta(M/N))$ is inert in $\mathcal{O}_N$, we can substitute the ramified part $\operatorname{Po}(M/N)_{\text{ram}}$ with a first cohomology group.

Finally, these maneuvers lead to a direct and more computable relation between the class number $h(N)$ and the sizes of first cohomology groups, which are more amenable to analysis.

\begin{remark}\label{Rema:Hilbert-identity}
Hilbert class field theory provides the trivial case of the above identity. Since $M/N$ is unramified by the Principal Ideal Theorem, and we generalized this in Section \ref{Sec:Gn-Van-Tr}, the term $\Po(M/N)_{\text{un}}$ will be eliminated. Moreover, since $M/N$ is unramified, both terms $\prod e_i(M/N)$ and $\#\Po(M/N)_{\text{ram}}$ will be 1.
\end{remark}

\section*{Notation}

For a number field $K$, let $\OO_K$ denote its ring of integers, $\Cl(K)$ its class group, and $ \OO_K^\times$ its unit group. For a finite extension $L/K$ of number fields, let $\Norm_{L/K}$ denote the norm map.

Let $\mathfrak{m}$ be a \emph{modulus} (formal product of finite primes and real places of $K$), and let $I_K^{\mathfrak{m}}$ denote the group of fractional ideals coprime to the finite part of $\mathfrak{m}$. Let $P_K^{\mathfrak{m}}$ be the subgroup of principal ideals generated by elements $\alpha \in K$ such that for finite $\mathfrak{p}^\varepsilon \mid \mathfrak{m}$, $\alpha \equiv 1 \pmod{\mathfrak{p}^\varepsilon}$; and for every real place $\tau$ in $\mathfrak{m}$, $\tau(\alpha) > 0$. The \emph{ray class group modulo $\mathfrak{m}$} is then defined as:
\[
\Cl^{\mathfrak{m}}(K) = I_K^{\mathfrak{m}} / P_K^{\mathfrak{m}}.
\]

For an abelian group $G$ and a prime $p$, let $G[p]$ and $G[p^n]$ denote the $p$-torsion and $p^n$-torsion subgroups of $G$, respectively. For a positive integer $m$, let $|G|_m$ denote the $m$-part of $|G|$. Finally, the \emph{exponent} of $G$ is defined by
\[
\exp(G) := \min \{ n \ge 1 \mid g^n = 1 \text{ for all } g \in G \} = \operatorname{lcm} \{ \operatorname{ord}(g) \mid g \in G \}.
\]

\section*{Acknowledgments}

The author gratefully thanks Taiwang Deng, Amir Jafari, Yongsuk Moon, Franz Lemmermeyer, Ali Partofard, Ali Rajaei, Kenneth Ribet, Ehsan Shahoseini, and Koji Shimizu for insightful discussions and valuable suggestions. In particular, I thank Professor Lemmermeyer for reviewing a completed draft at my request and for flagging issues, which I have since addressed. This work was supported by a grant from the Beijing Institute of Mathematical Sciences and Applications, and by the Beijing National Natural Science Foundation of China (BNSFC).

\section{Background}

\subsection{Relative Polya Group for Relative Extension $F/K$}

\begin{definition}[\cite{Maarefparvar2020}, Definition 2.3]
Let $F/K$ be a finite extension of number fields. The \emph{relative Polya group} of $F$ over $K$ is the subgroup of $\Cl(F)$ generated by the classes of the ideals $\Pi_{\mathfrak{P}^f}(F/K)$, where $\mathfrak{P}$ is a prime ideal of $K$, $f$ is a positive integer, and $\Pi_{\mathfrak{P}^f}(F/K)$ is defined as follows:
\[
\Pi_{\mathfrak{P}^f}(F/K) = \prod_{\substack{\mathfrak{M} \in \operatorname{Max}(\OO_F) \\ \Norm_{F/K}(\mathfrak{M}) = \mathfrak{P}^f}} \mathfrak{M}.
\]
We denote the relative Polya group of $F$ over $K$ by $\Po(F/K)$. In particular, $\Po(F/\mathbb{Q}) = \Po(F)$ and $\Po(K/K) = \Cl(K)$.
\end{definition}
We further define:
\begin{itemize}
\item $\Po(M/N)_{\text{ram}}$: the subgroup of $\Po(M/N)$ generated by classes of ramified prime ideals.
\item $\Po(M/N)_{\text{un}}$: the subgroup of $\Po(M/N)$ generated by classes of unramified prime ideals.
\end{itemize}

This yields the decomposition:
\[
\Po(M/N) \cong \Po(M/N)_{\text{ram}} \oplus \Po(M/N)_{\text{un}}.
\]

Note that when $N = \mathbb{Q}$, we have $\Po(M/\mathbb{Q})_{\text{un}} = \{1\}$, so $\Po(M/\mathbb{Q}) = \Po(M/\mathbb{Q})_{\text{ram}}$.
We have the following lemmas:

\begin{lemma}\label{lem:relative-polya-exact}
Let $F/K$ be a finite Galois extension of number fields with Galois group $G$. Then the following sequence is exact:
\begin{multline}\label{eq:relative-polya-exact}
1 \longrightarrow \operatorname{Ker}(\epsilon_{M/N}) \longrightarrow H^1(\Gal(M/N), \UU_M) \\
\longrightarrow \bigoplus_{\mathfrak{P}} \mathbb{Z}/e_{\mathfrak{P}}\mathbb{Z} \longrightarrow \frac{\operatorname{Po}(M/N)}{\epsilon_{M/N}(\Cl(N))} \longrightarrow 1.
\end{multline}
For a finite extension $F/K$ of number fields, $\epsilon_{F/K}: \Cl(K) \to \Cl(F)$ denotes the transfer of ideal classes induced by the morphism $j_{F/K}: \mathfrak{a} \in I(K) \mapsto \mathfrak{a}\OO_F \in I(F)$.
\end{lemma}

We also have:

\begin{lemma}[\cite{Zantema1982}, §3, p. 9]\label{lem:absolute-polya-exact}
Let $E/\mathbb{Q}$ be a Galois extension with Galois group $G$. Denote the ramification index of a prime $p$ in $E$ by $e_p$. Then the following sequence is exact:
\begin{equation}\label{eq:absolute-polya-seq}
1 \longrightarrow H^1(G, \UU_E) \longrightarrow \bigoplus_{p \text{ prime}} \mathbb{Z}/e_p\mathbb{Z} \longrightarrow \Po(E) \longrightarrow 1.
\end{equation}
\end{lemma}

We now review the method used to derive the exact sequence \eqref{lem:relative-polya-exact} (see also \cite[§2]{Maarefparvar2020}). Let $F/K$ be a Galois extension with Galois group $G$. It is straightforward to verify that the set of all ideals $\Pi_{\mathfrak{P}^f}(F/K)$ forms a free generating set for the ambiguous ideals $I(F)^G$. On the other hand, we have $P(F)^G = I(F)^G \cap P(F)$. Therefore, in this case, the relative Polya group is given by $\Po(F/K) = I(F)^G / P(F)^G$. Following Zantema's method \cite[§3]{Zantema1982}, the relative Polya group for the extension $F/K$ is defined as in Definition 2.1. We define the map $\psi$ over the basis elements $\Pi_{\mathfrak{P}^f}(F/K)$ of $I(F)^G$ as follows:
\begin{align*}
\psi: I(F)^G &\to \bigoplus_{\mathfrak{P} \text{ prime of } K} \mathbb{Z}/e_{\mathfrak{P}}\mathbb{Z} \\
\psi\left(\Pi_{\mathfrak{P}^f}(F/K)^t\right)_{\mathfrak{M}} &\coloneqq t \pmod{e_{\mathfrak{P}}},
\end{align*}
where $\Norm_{F/K}(\mathfrak{M}) = \mathfrak{P}^f$, $t$ is an integer, and $e_{\mathfrak{P}}$ denotes the ramification index of $\mathfrak{P}$ in $F/K$. The map $\psi$ is clearly a group epimorphism, and one can show that $\ker(\psi) = I(K)$. This yields the exact sequence:
\[
1 \longrightarrow I(K) \longrightarrow I(F)^G \longrightarrow \bigoplus_{\mathfrak{p} \text{ prime of } K} \mathbb{Z}/e_{\mathfrak{p}}\mathbb{Z} \longrightarrow 1.
\]

Next, consider the short exact sequence:
\[
1 \longrightarrow \UU_F \longrightarrow F^* \longrightarrow P(F) \longrightarrow 1.
\]
Taking cohomology and applying Hilbert's Theorem 90, we obtain:
\[
1 \longrightarrow \UU_K \longrightarrow K^* \longrightarrow P(F)^G \longrightarrow H^1(\Gal(F/K), \UU_F) \longrightarrow 1.
\]
Equivalently, the sequence simplifies to:
\[
1 \longrightarrow P(K) \longrightarrow P(F)^G \longrightarrow H^1(\Gal(F/K), \UU_F) \longrightarrow 1.
\]

Finally, define the map:
\begin{align*}
\epsilon: \Cl(K) &\to I(F)^G / P(F)^G \\
\overline{\mathfrak{a}} &\mapsto j(\overline{\mathfrak{a}}).
\end{align*}

Combining these results, we construct the commutative diagram:
\[
\begin{tikzcd}[column sep=small]
 & 0 \arrow[d] &0 \arrow[d]\\
0 \arrow[r] & P(K) \arrow[d] \arrow[r] & I(K) \arrow[d] \arrow[r] & \mathrm{Cl}(K) \arrow[d] \arrow[r] & 0 \\
0 \arrow[r] & P(F)^G \arrow[r] \arrow[d] & I(F)^G \arrow[r] \arrow[d] & I(F)^G / P(F)^G \arrow[r] & 0 \\
0 \arrow[r] & H^1(G(F/K), U_F) \arrow[d] \arrow[r] & \bigoplus_{\mathfrak{p} \mid \operatorname{disc}(F/K)} \mathbb{Z}/e_{\mathfrak{p}}\mathbb{Z} \arrow[d] \\
& 0 & 0
\end{tikzcd}
\]

In the next section, we show that an unramified prime ideal \(\mathfrak{p} \in \operatorname{Cl}(N)\) of order \(p^n\) capitulates in \(M\) under certain conditions, where \(M\) is a Galois extension of the number field \(N\) of degree \(p^l\) with \(n \leq l\).

\section{Generalizing the Vanishing Transfer Map Theorem}\label{Sec:Gn-Van-Tr}

The first step is to generalize the usual transfer map from $G$ to the quotient $G/H$. We then prove the capitulation lemma (Lemma \ref{lem:vanishing-transfer}).

The transfer map is a group homomorphism
\[
\operatorname{Ver}(G \to H): G/G' \to H/H'
\]
defined as follows: Let $G$ be a finite group and $H$ a subgroup (not necessarily normal). Choose left coset representatives $g_1, \ldots, g_n$ for $G/H$, so
\[
G = \bigsqcup_{i=1}^n g_i H.
\]
Define a section map $\phi: G \to \{g_1, \ldots, g_n\}$ by $\phi(g) = g_k$ where $g_k$ is the representative satisfying $g \in g_k H$. For $g \in G$, the \textbf{transfer element} is
\[
\operatorname{Ver}(g) = \prod_{i=1}^n \left(\phi(gg_i)^{-1} gg_i\right) \in H.
\]
This map factors through $G/G'$ and induces the homomorphism $\operatorname{Ver}(G \to H): G/G' \to H/H'$.

\begin{definition}\label{def:restricted-transfer}
Suppose $H$ is a normal subgroup of $G$ containing the derived subgroup $G'$. The \textbf{restricted transfer map}
\[
\widetilde{\operatorname{Ver}}: G/H \to H/H'
\]
is defined when the transfer map $\operatorname{Ver}(G \to H)$ vanishes on $H$, that is, $\operatorname{Ver}(G \to H)(h) \in H'$ for all $h \in H$. This condition holds, for example, when $\operatorname{exp}(G)$ divides $[G:H]$, when $H$ is central in $G$, or when $G$ is abelian, as seen from the congruence $\operatorname{Ver}(G \to H)(h) \equiv h^{[G:H]} \pmod{H'}$ for $h \in H$.
\end{definition}

\begin{remark}
In the present setup $G=\operatorname{Gal}(L/N)\cong(\mathbb{Z}/p^2)^2$, $\operatorname{Gal}(L/H)=pG\cong(\mathbb{Z}/p)^2$, and $\operatorname{Gal}(M/N)\cong\mathbb{Z}/p$. Take the subgroup $\operatorname{Gal}(L/H)=pG$ in Definition~\ref{def:restricted-transfer}. It is abelian (so $H'=1$), and $[G:pG]=p^2$, so $\operatorname{Ver}(g)=g^{p^2}=1$ for all $g\in G$. Hence $\widetilde{\operatorname{Ver}}:G/pG\to pG$ vanishes identically. This does not hold for the transfer into $\operatorname{Gal}(L/M)$, where $\operatorname{Ver}(t)=t^{p}$ can be nontrivial.
\end{remark}

\begin{lemma}\label{lem:vanishing-transfer}
Let $H$ be a normal subgroup of abelian group $G$. If $exp(G)$ divides $[G:H]$, then
\[
\widetilde{\operatorname{Ver}}: G/H \to H/H'
\]
vanishes.
\end{lemma}

We prove this theorem similarly to the usual vanishing transfer map. This requires the following commutative diagram:
\[
\begin{tikzcd}
G/H \arrow[r, "\widetilde{\operatorname{Ver}}"] \arrow[d, "\delta|_{H}"] & H/H' \arrow[d, "\delta"] \\
\frac{I_G/I_G^2}{\delta(H/G')} \arrow[r, "S'"] & \frac{I_H + I_G I_H}{I_G I_H}
\end{tikzcd}
\]

We now define the analogous homomorphism for the quotient $G/H$ in place of $G/G'$. This requires defining an $S'$-homomorphism to make this diagram commute. To that end, we first recall the definitions of the original maps for $G/G'$, and then generalize them to the quotient $G/H$, where $G' \subset H$.

\subsection{Generalizing the Maps $\delta$ and $S$}

The group ring $\mathbb{Z}[G]$ is defined as
\[
\mathbb{Z}[G] = \left\{\sum_{\sigma \in G} n_\sigma \sigma \mid n_\sigma \in \mathbb{Z}\right\}.
\]
Let $I_G$ denote the augmentation ideal, defined as the kernel of the augmentation map:
\[
i: \mathbb{Z}[G] \to \mathbb{Z}, \quad \sum_{\sigma \in G} n_\sigma \sigma \mapsto \sum_{\sigma \in G} n_\sigma.
\]
For any subgroup $H \subseteq G$, the set $\{\tau - 1 \mid \tau \in H \setminus \{1\}\}$ forms a $\mathbb{Z}$-basis for $I_H$, and we have the inclusion $I_H \subseteq I_G$.

\begin{definition}[Isomorphism $\delta$]\label{def:delta-isomorphism}
The map $\delta: G/G' \to I_G/I_G^2$ is defined by:
\[
\delta(\sigma G') = (\sigma - 1) + I_G^2 \quad \text{for} \quad \sigma \in G.
\]
This is a canonical isomorphism of abelian groups. For a subgroup $H \subseteq G$, the following map is an isomorphism as well (see \cite[Ch VI, §7, Lemma 7.7]{Neukirch2013}):
\[
\delta: H/H' \to (I_H + I_G I_H)/I_G I_H
\]
is defined by $\delta(hH') = (h - 1) + I_G I_H$ for $h \in H$.
\end{definition}

\begin{definition}[Restricted $\delta$]\label{def:restricted-delta}
Suppose $H$ is a normal subgroup containing $G'$. Then $G/H \cong (G/G')/(H/G')$, and we define:
\[
\delta|_{H}: G/H \to \frac{I_G/I_G^2}{\delta(H/G')}
\]
as the isomorphism induced by $\delta$.
\end{definition}

\begin{definition}[Norm Map $S$]\label{def:norm-map-S}
Let $\{g_1, \ldots, g_n\}$ be coset representatives of $H$ in $G$, and $N = [g_1] + \cdots + [g_n] \in \mathbb{Z}[G]$ the norm element. The map $S: I_G/I_G^2 \to (I_H + I_G I_H)/I_G I_H$ is defined by:
\[
S(x) = x \cdot ([g_1] + \cdots + [g_n]), \quad \text{mod} \quad I_G I_H
\]
where $x \in I_G/I_G^2$. This makes the following diagram commute \cite[Ch VI, §7, Lemma 7.7]{Neukirch2013}:
\[
\begin{tikzcd}
G/G' \arrow[r, "\operatorname{Ver}"] \arrow[d, "\delta"] & H/H' \arrow[d, "\delta"] \\
I_G/I_G^2 \arrow[r, "S"] & (I_H + I_G I_H)/I_G I_H
\end{tikzcd}
\]
\end{definition}

\begin{definition}[Induced Map $S'$]\label{def:induced-map-Sprime}
Assume $G' \subset H$ and $|H| \mid [G:H]$. The map $S$ factors through $\delta(H/G')$ to give:
\[
S': \frac{I_G/I_G^2}{\delta(H/G')} \to (I_H + I_G I_H)/I_G I_H
\]
defined by $S'([x]) = S(x)$, where $[x]$ denotes the class of $x$ modulo $\delta(H/G')$.
\end{definition}

 $S'$ is well-defined because $S(\delta(h)) = 0$ for all $h \in H$. This follows from the commutative diagram in Definition \ref{def:norm-map-S}, which gives:
\[
S(\delta(h)) = \delta(\operatorname{Ver}(h)) = \delta(0) = 0,
\]
since $\operatorname{Ver}(h) = 0$ whenever $|H|$ divides $[G:H]$. Thus $S'$ is well-defined.

\begin{lemma}\label{lem:commutative-diagram-transfer}
Let $G$ be an abelian group, and let $H$ be a subgroup of $G$ such that $\exp(G)$ divides $[G:H]$. Then the following diagram commutes:
\[
\begin{tikzcd}
G/H \arrow[r, "\widetilde{\operatorname{Ver}}"] \arrow[d, "\delta|_{H}"'] & H/H' \arrow[d, "\delta"] \\
\frac{I_G/I_G^2}{\delta(H/G')} \arrow[r, "S'"'] & \frac{I_H + I_G I_H}{I_G I_H}
\end{tikzcd}
\]
\end{lemma}

\begin{proof}
Since through $\operatorname{Ver}(G \to H)$ the above diagram is commutative. Now for every $g \in G \bmod H$ and $a \in I_G/I_G^2 \bmod \delta(H/G')$ obviously diagram commutes as well.
\end{proof}

\subsection{Proof of Lemma \ref{lem:vanishing-transfer}}

\begin{proof}
Dividing out by $H'$, we may assume without loss of generality that $H$ is abelian. By the fundamental theorem of finite abelian groups, we write
\[
G/H \cong \prod_{i=1}^m \mathbb{Z}/e_i\mathbb{Z}.
\]
Choose representatives $f_i \in G$ lifting generators of the respective cyclic factors, and set $h_i := f_i^{-e_i} \in H$.

Setting $\delta x := x - 1$, we recall the standard identities $\delta(xy) = \delta x + \delta y + \delta x \delta y$ and $\delta(x^{-1}) = -x^{-1} \delta x$. Applying Lemma~\ref{lem:commutative-diagram-transfer} gives
\[
0 = \delta(f_i^{e_i} h_i) = \delta(f_i^{e_i}) (1 + \delta h_i) + \delta h_i.
\]
Because the transfer map $\operatorname{Ver}(G \to H)$ takes values in $H'$ under our hypotheses ($\exp(G) \mid [G:H]$ and $G$ is abelian), we have $h_i \in H'$, so that $\delta h_i = 0$. Consequently,
\[
0 = \delta(f_i^{e_i} h_i) = \delta(f_i) \mu_i
\]
for some elements $\mu_i \in \mathbb{Z}[G]$ satisfying $\mu_i \equiv e_i \pmod{I_G}$.

Now consider the product $\mu := \prod_{i=1}^m \mu_i$. Because each $\mu_i$ descends to $\mathbb{Z}[G/H] \cong \mathbb{Z}[G]/I_H\mathbb{Z}[G]$ and $G/H$ is abelian, $\mu$ is well-defined modulo $I_H\mathbb{Z}[G]$. We then express:

\[
\mu \equiv \sum_{g_j \in R} n_j g_j \pmod{I_H\mathbb{Z}[G]},
\]
where $R$ is a complete set of coset representatives for $G/H$.

For any $\overline{g} \in G/H$, left-multiplication gives:
\[
\overline{g} \mu \equiv \sum_{g_j \in R} n_j g_j \overline{g} \pmod{I_H\mathbb{Z}[G]}.
\]
By $G/H$-invariance of $\mu$, all coefficients $n_j$ equal some $k \in \mathbb{Z}$. Thus:
\[
\mu \equiv k \sum_{g_j \in R} g_j \equiv k \left(\sum_{g_j \in R} (g_j - 1) + [G:H]\right) \pmod{I_H\mathbb{Z}[G]}.
\]

From the product expansion, $\mu \equiv [G:H] \pmod{I_G}$. Comparing with:
\[
k \left(\sum_{g_j \in R} (g_j - 1) + [G:H]\right) \equiv k[G:H] \pmod{I_G},
\]
we obtain $k[G:H] \equiv [G:H] \pmod{I_G}$, forcing $k = 1$. Therefore:
\[
\mu \equiv \sum_{g_j \in R} g_j \pmod{I_G I_H}.
\]

For $g \in G/H$, Lemma \ref{lem:commutative-diagram-transfer} yields:
\[
S'(\delta g) \equiv \delta g \cdot \sum_{g_j \in R} g_j \equiv (\delta g)\mu \equiv 0 \pmod{I_G I_H},
\]
This confirms that \(\widetilde{\operatorname{Ver}}\) vanishes. 
(An alternative proof can be obtained by replacing \(G'\) with a subgroup \(H\) containing \(G'\) in the argument of Theorem 7.6 in \cite[Ch.~VI]{Neukirch1986}. 
The only requirement is that the Verlagerung map \(\operatorname{Ver}: G/H \to H\) is a well-defined homomorphism, a condition which is satisfied by the hypothesis of our lemma.)
\end{proof}

\section{Application of Lemma \ref{lem:vanishing-transfer} to Number Fields}

We now apply Lemma \ref{lem:vanishing-transfer} in the context of number fields. Specifically, we examine the consequences for towers of number fields over a real quadratic field \(N\).

\begin{lemma}\label{lem:number-field-application}
Let $N\subset H\subset L$ be number fields with $L/N$ abelian, $\exp(\Gal(L/N))\mid[H:N]$, $H/N$ Galois, and $\mathfrak{m}$ a modulus divisible by the conductor of $L/N$. Then the diagram
\[
\begin{tikzcd}
\Gal(H/N) \arrow[r, "\widetilde{\operatorname{Ver}}"] \arrow[d, "\mathrm{Art}"] & \Gal(L/H) \arrow[d, "\cong"] \\
I_N^{\mathfrak{m}} \arrow[r, "\varphi"] & I_H^{\mathfrak{m}} / \mathcal{P}_H^{\mathfrak{m}} \Norm_{L/H}(I_L^{\mathfrak{m}})
\end{tikzcd}
\]
commutes (up to the Artin map being defined on $I_N^{\mathfrak m}$), $\varphi(\mathfrak{a})=\mathfrak{a}\OO_H$ modulo the norm group, and $\varphi$ is trivial on ideals of $N$ prime to $\mathfrak m$ (as $\widetilde{\operatorname{Ver}}=0$).
\end{lemma}

\begin{proof}
Let $G=\Gal(L/N)$ and $H'=\Gal(L/H)$. As $G$ is abelian and $\exp(G)\mid[G:H']=[H:N]$, $\operatorname{Ver}(g)=g^{[G:H']}=1$, so $\widetilde{\operatorname{Ver}}:G/H'\to H'$ is well defined and zero (Lemma~\ref{lem:vanishing-transfer}). Commutativity is the compatibility of transfer with extension of ideals in class field theory \cite[Theorem 6.13(b)]{Neukirch1986} together with Artin reciprocity \cite[VI, §7, Thm 7.1]{Neukirch2013}; the left map is the Artin map for $H/N$ and the right one that of $L/H$.
\end{proof}

To study ideal classes through Lemma~\ref{lem:number-field-application}, take $I$ with no ramified prime and with every prime factor inert in $M/N$ (Setup~\ref{setup:tower}). Then $I$ is not in $\mathcal{P}_N^{\mathfrak m}\Norm_{M/N}(\mathcal I_M^{\mathfrak m})$, whereas $I\OO_H$ lies in $\mathcal{P}_H^{\mathfrak m}\Norm_{L/H}(\mathcal I_L^{\mathfrak m})$ by Theorem~\ref{thm:main-technical}(3).



In the rest of this section, we take conditions for ideal class $[I]$ in $Cl(N) $ such that they are satisfies for $M/N$ and leads automaticaly to $L/M$ such that the above Lemma \ref{lem:number-field-application} leads to the capitulation of $I\mathcal{O}_M$ in $Cl(M)$ for $M$ satisfies the conditions of above Lemma \ref{lem:number-field-application}. We find these conditions in Theorem \ref{thm:capitulation}.

The following lemma demonstrates that in a cyclic extension $ K/k$ with intermediate fields, an ideal of $k$ can be inert only in the top layer $ K/k$, where $E$ is the maximal proper subextension.

\begin{lemma}\label{lem:inert-subextension}
Let \( K/k \) be a cyclic extension of number fields. Let \( \mathfrak{P} \) be an unramified prime ideal in \( \OO_K \) lying over \( \mathfrak{p} \) in \( \OO_k \). Let \( G_{\mathfrak{P}} \) be the decomposition group of \( \mathfrak{P} \) over \( k \). Then the fixed field \( E = K^{G_{\mathfrak{P}}} \) is the maximal intermediate extension in which \( \mathfrak{p} \) is inert.
\end{lemma}

\begin{proof}
Let \( G = \Gal(K/k) \). Since \( K/k \) is cyclic, \( G \) is cyclic. As \( \mathfrak{P} \) is unramified, the Frobenius element \( \left(\frac{K/k}{\mathfrak{P}}\right) \) generates the cyclic decomposition group \( G_{\mathfrak{P}} \), which has order \( f = f(\mathfrak{P}/\mathfrak{p}) \) (the residue degree).

Let \( E = K^{G_{\mathfrak{P}}} \) be the fixed field. Then \( [K:E] = |G_{\mathfrak{P}}| = f \). In the extension \( K/E \), the prime \( \mathfrak{q} = \mathfrak{P} \cap \OO_E \) has residue degree \( f(K/E) = [K:E] = f \), which means \( \mathfrak{q} \) is inert in \( K/E \).

Moreover, \( E \) is maximal with this property: if \( E \subset F \subset K \) and \( \mathfrak{p} \) is inert in \( K/F \), then the decomposition group of \( \mathfrak{P} \) over \( F \) would be all of \( \Gal(K/F) \), forcing \( F \subset E \).
\end{proof}
\begin{definition}\label{ass:main-assumptions}
Let $N$ be a real quadratic field and $I$ an ideal with unramified prime factors. We say that $M$ is a \emph{relatively $I$-proper extension} of $N$ if there is a tower $N\subset M\subset H\subset L$ as in Setup~\ref{setup:tower} such that every prime factor of $I$ is inert in $M/N$. Then $N\subset H\subset L$ satisfies the hypotheses of Lemma~\ref{lem:number-field-application}.
\end{definition}

\begin{lemma}\label{lem:inert-in-tower}
Let $\mathfrak{q}_N$ be a prime of $N$ inert in a relatively proper extension $M/N$ (Setup~\ref{setup:tower}). Then the primes of $H$ above $\mathfrak q_N$ have residue degree $p$ over $\mathfrak q_N$ and are not split in $L/H$ completely.
\end{lemma}

\begin{proof}
Frobenius $\sigma\notin\Gal(L/M)$ has order $p^2$ (since $\sigma\notin pG$ and $\exp G=p^2$), so $f(L/N)=p^2<p^4$. Apply Theorem~\ref{thm:main-technical}(1),(2).
\end{proof}

\begin{corollary}\label{cor:non-norm}
In Setup~\ref{setup:tower}, let $\mathfrak q_N$ be a prime inert in $M/N$. Then $\mathfrak q_N\OO_H\in\mathcal P_H^{\mathfrak m}\Norm_{L/H}(I_L^{\mathfrak m})$, while no prime of $H$ above $\mathfrak q_N$ splits completely in $L/H$.
\end{corollary}

\begin{proof}
Theorem~\ref{thm:main-technical}(2),(3).
\end{proof}

\begin{theorem}\label{thm:capitulation}
Let $\mathfrak q$ be a prime of $N$ unramified in $L$ and inert in $M/N$, where $M$ is relatively $\mathfrak q$-proper (Definition~\ref{ass:main-assumptions}). Then $[\mathfrak q]\notin \mathcal P_N^{\mathfrak m}\Norm_{M/N}(\mathcal I_M^{\mathfrak m})$ but $\varphi(\mathfrak q)=\mathfrak q\OO_H$ is trivial in $\mathcal I_H^{\mathfrak m}/\mathcal P_H^{\mathfrak m}\Norm_{L/H}(\mathcal I_L^{\mathfrak m})$.
\end{theorem}

\begin{proof}
The first claim is the decomposition theorem, as $\mathfrak q$ is inert in $M/N$. The second is Corollary~\ref{cor:non-norm} (equivalently Lemma~\ref{lem:number-field-application} with $\widetilde{\operatorname{Ver}}=0$).
\end{proof}

This inspires the following definition:

\begin{definition}\label{def:Cl_M}
Let \( N \) be a real quadratic field and \( M \) a Galois extension of \( N \) of degree \( p^n \). Suppose \( \mathfrak{a} \) is an ideal of \( \mathcal{O}_N \) whose class has order \( p^n \) in \( \operatorname{Cl}(N) \), and that \( M/N \) is a relatively \( \mathfrak{a} \)-proper extension of degree \( p^n \). 

We define \( \operatorname{Cl}_M(N) \) as the subset of \( \operatorname{Cl}(N) \) consisting of ideal classes \( [\mathfrak{b}] \) that capitulate in \( \operatorname{Cl}(M) \).
\end{definition}
\begin{remark}\label{rem:}
Let \( \mathfrak{q}_1, \dots, \mathfrak{q}_s \) be prime ideals in \( N \) whose classes in \( \operatorname{Cl}(N) \) have order \( p^n \). 
Assume that \( M/N \) is a cyclic extension of degree \( p^d \) which is relatively \(\mathfrak{q}_1\)-proper, with \( n \leq d \). Then, we will show that, only the prime \(\mathfrak{q}_1\) capitulates in \( \operatorname{Cl}(M) \).

\end{remark}


\begin{remark}\label{rem:coprime-order-splits}
If \( \mathfrak{q} \) is a prime in \( \Cl(N) \) whose "\textit{order}" is relatively prime to \( |\Gal(M/N)| \), then the Artin symbol \( \left(\frac{M/N}{\mathfrak{q}}\right) \) must be trivial. By class field theory, this implies that \( \mathfrak{q} \) splits completely in \( M/N \).
\end{remark}




\section{Ramified Part of the Relative Pólya Group}

In Section~2 we recalled the definition of the relative Pólya group. 
In what follows, we will aim to replace \(\Po(M/N)\) by the simpler group \(\Po(M)\), which allows us to work entirely with first cohomology.

In the next proposition we show that, for a suitably chosen extension \( M \), the ramified part \( \operatorname{Po}(M/N)_{\mathrm{ram}} \) is isomorphic to \( \operatorname{Po}(M) \) up to a factor which is a power of~\(2\). We then recall the notion of a \emph{proper} extension.

\begin{proposition}\label{prop:ramified-exact-sequence}
Let \( N \) be a real quadratic field and let \( M \) be a Galois extension of both \( N \) and \( \mathbb{Q} \). Assume that 
\[
\operatorname{Norm}_{N/\mathbb{Q}}\!\bigl(\mathfrak{d}(M/N)\bigr)
\]
is inert in \( N \). Then the following sequence is exact:
\begin{equation}\label{eq:ramified-exact-seq0}
1 \longrightarrow \operatorname{Po}(N) \longrightarrow \operatorname{Po}(M) \longrightarrow \operatorname{Po}(M/N)_{\mathrm{ram}} \longrightarrow 1 .
\end{equation}
\end{proposition}

\begin{proof}
The groups \( \Po(M) \) and \( \Po(M/N)_{\text{ram}} \) correspond to ramified primes in \( M/\mathbb{Q} \) and \( M/N \), respectively.

\begin{itemize}
\item Primes that ramify only in \( M/N \) contribute identical terms to the subgroups generated by \(\Pi_p(M)\) and \(\Pi_p(M/N)\); by the assumption, only primes inert in \( \mathcal{O}_N \) can ramify in \( M/N \).

\item Primes that ramify in \( M/\mathbb{Q} \) but not in \( M/N \) generate a subgroup \( \Po'(M) \subset \Po(M) \). Consider a prime \( p \) such that \( p\mathcal{O}_N = \mathfrak{p}^2 \) (i.e., \( p \) is ramified in \( N/\mathbb{Q} \)) and assume that \( \mathfrak{p} \) is either inert or splits in \( M/N \). Then \( \Pi_p(M) = \mathfrak{p}^2 \mathcal{O}_M \). This yields an injective homomorphism
\[
\Po'(M) \hookrightarrow \Po(N), \tag{6}
\]
induced by the norm map \( \Norm_{M/N}\colon \Pi_{\mathfrak{p}}(M) \to \Pi_p(N) \). Conversely, there is an injective homomorphism
\[
\Po(N) \hookrightarrow \Po'(M), \tag{7}
\]
sending \( \mathfrak{p} \in \Po(N) \) to \( \mathfrak{p} \mathcal{O}_M \). Consequently, \( \Po(N) \cong \Po'(M) \).
\end{itemize}
This yields the following exact sequence, which splits:

\begin{equation}\label{eq:ramified-exact-seq}
1 \longrightarrow \operatorname{Po}(N) \longrightarrow \operatorname{Po}(M) \longrightarrow \operatorname{Po}(M/N)_{\mathrm{ram}} \longrightarrow 1 .
\end{equation}
\end{proof}

\begin{corollary}\label{cor:polya-quotient-power-two}
The quotient
\[
\frac{|\Po(M)|}{|\Po(M/N)_{\text{ram}}|} = |\Po(N)|
\]
is a power of two.
\end{corollary}

\begin{proof}
This follows from Proposition \ref{prop:ramified-exact-sequence} and the exact sequence:
\[
1 \to H^1(G, \UU_N) \to \bigoplus_{p \mid \disc(N/\mathbb{Q})} \mathbb{Z}/e_p\mathbb{Z} \to \Po(N) \to 1,
\]
where \( |\Po(N)| \) is a power of two.
\end{proof}

\begin{definition}\label{def:proper-field}
Let \( N \) be a number field and let \( I \) be an ideal of \( \mathcal{O}_N \). 
We say that an extension \( M/N \) is \emph{proper for \( I \)} if the following conditions hold:
\begin{enumerate}[label=(\alph*)]
    \item \( M/\mathbb{Q} \) is Galois.
    \item \( M/N \) is relatively \( I \)-proper (in the sense of Definition~\ref{ass:main-assumptions}).
    \item The norm \( \operatorname{Norm}_{N/\mathbb{Q}}(\mathfrak{d}(M/N)) \) is inert in \( N \).
\end{enumerate}
\end{definition}


In the next section, we show how to eliminate the factor \( \operatorname{Po}(M/N) \) and the relative ramification term in the formula
\begin{equation}\label{eq:cardinality-relation-whole}
|\Cl(N)| \cdot \biggl| \bigoplus_{\mathfrak{p} \mid \mathfrak{d}(M/N)} \mathbb{Z}/e_{\mathfrak{p}}\mathbb{Z} \biggr| = 
\bigl| H^1(\operatorname{Gal}(M/N), \mathcal{O}_M^{\times}) \bigr| \cdot |\operatorname{Po}(M/N)|,
\end{equation}
when \( M/N \) is a proper extension.

\section{Removing \( \operatorname{Po}(M/N) \) from Relation~\eqref{eq:cardinality-relation-whole}}

In this section we combine the results of the preceding sections to obtain a new cardinality formula.

Let \( N \) be a quadratic field such that \(\Cl(N)\) contains \(k\) elements of order \(p^n\). Let \(\mathfrak{q}_1,\dots,\mathfrak{q}_s\) be prime ideals of \(\mathcal{O}_N\) whose classes in \(\operatorname{Cl}(N)\) have order \(p^n\). 
Assume that \( M/N \) is a cyclic extension of degree \( p^n \) that is proper for the ideal \(\mathfrak{q}_1\). Then the following equality holds modulo factors coprime to \( p \):
\begin{multline}\label{eq:new-cardinality}
|\Cl(N)| \cdot \biggl| \bigoplus_{\mathfrak{p} \mid \mathfrak{d}(M/N)} \mathbb{Z}/e_{\mathfrak{p}}\mathbb{Z} \biggr| \\
   =_{p^n} \bigl| H^1(\operatorname{Gal}(M/N), \mathcal{O}_M^{\times}) \bigr|
     \cdot |\operatorname{Po}(M/N)_{\mathrm{un}}| \cdot |\operatorname{Po}(M/N)_{\mathrm{ram}}|,
\end{multline}
where the notation \( =_{p^n} \) means that the two sides are equal when only their \( p^n \)-parts are compared.

Rearranging terms yields
\begin{multline}\label{eq:cl-m-n-quotient}
|\Cl(N)| =_{p^n} 
\frac{\bigl| H^1(\operatorname{Gal}(M/N), \mathcal{O}_M^{\times}) \bigr| 
      \cdot |\operatorname{Po}(M/N)_{\mathrm{ram}}|
      \cdot |\operatorname{Po}(M/N)_{\mathrm{un}}|}
     {\displaystyle\prod_{\mathfrak{p} \mid \mathfrak{d}(M/N)} e_{\mathfrak{p}}},
\end{multline}
where \( e_{\mathfrak{p}} \) denotes the ramification index of \( \mathfrak{p} \) in \( M/N \).

\begin{lemma}\label{lem:polya-ram-size-H1}
Under the assumptions above, let \( M/N \) be a proper extension for the ideal \( \mathfrak{q}_1 \). Then
\begin{equation}\label{eq:polya-ram-size-H1}
\frac{\bigl|\operatorname{Po}(M/N)_{\mathrm{ram}}\bigr|}
{\displaystyle\prod_{\mathfrak{p} \mid \mathfrak{d}(M/N)} e_{\mathfrak{p}}}
\;=\;
\frac{\bigl|H^1(\operatorname{Gal}(N/\mathbb{Q}), \mathcal{O}_N^{\times})\bigr|}
{\bigl|H^1(\operatorname{Gal}(M/\mathbb{Q}), \mathcal{O}_M^{\times})\bigr|}.
\end{equation}
\end{lemma}

\begin{proof}
Proposition~\ref{prop:ramified-exact-sequence} gives the identity
\begin{equation}\label{eq:polya-ram-size}
\bigl|\operatorname{Po}(M/N)_{\mathrm{ram}}\bigr| = \frac{|\operatorname{Po}(M)|}{|\operatorname{Po}(N)|}.
\end{equation}
Recall that \( \operatorname{Po}(N) \) is a power of \(2\) because \( N \) is a real quadratic field. For the extension \( M/\mathbb{Q} \) we have the exact sequence
\[
1 \longrightarrow H^1\!\bigl(\operatorname{Gal}(M/\mathbb{Q}), \mathcal{O}_M^{\times}\bigr) \longrightarrow 
\bigoplus_{p \mid \mathfrak{d}(M/\mathbb{Q})} \mathbb{Z}/e_p\mathbb{Z} \longrightarrow \operatorname{Po}(M) \longrightarrow 1,
\]
and therefore
\[
|\operatorname{Po}(M)| = 
\frac{\displaystyle\Bigl|\bigoplus_{p \mid \mathfrak{d}(M/\mathbb{Q})} \mathbb{Z}/e_p\mathbb{Z}\Bigr|}
{\bigl|H^1(\operatorname{Gal}(M/\mathbb{Q}), \mathcal{O}_M^{\times})\bigr|},
\]
Similarly, 
\[
|\operatorname{Po}(N)| = 
\frac{\displaystyle\Bigl|\bigoplus_{p \mid \mathfrak{d}(N/\mathbb{Q})} \mathbb{Z}/e_p\mathbb{Z}\Bigr|}
{\bigl|H^1(\operatorname{Gal}(N/\mathbb{Q}), \mathcal{O}_N^{\times})\bigr|},
\]
Since \( N/\mathbb{Q} \) is quadratic,$M$ is Galois over $N$, Consequently,
\[
\Biggl|
\frac{\displaystyle\bigoplus_{p \mid \mathfrak{d}(M/\mathbb{Q})} \mathbb{Z}/e_p\mathbb{Z}}
{\displaystyle\bigoplus_{\mathfrak{p} \mid \mathfrak{d}(M/N)} \mathbb{Z}/e_{\mathfrak{p}}\mathbb{Z}}
\Biggr|
\;=\;
\Bigl| \bigoplus_{p \mid \mathfrak{d}(N/\mathbb{Q})} \mathbb{Z}/e_p\mathbb{Z} \Bigr|.
\]

Now substitute the expressions for \( |\operatorname{Po}(M)| \) and \( |\operatorname{Po}(N)| \) into \eqref{eq:polya-ram-size}. After canceling the common factor
\[
\Bigl| \bigoplus_{p \mid \mathfrak{d}(N/\mathbb{Q})} \mathbb{Z}/e_p\mathbb{Z} \Bigr|,
\]
 we obtain precisely \eqref{eq:polya-ram-size-H1}.
\end{proof}

Lemma~\ref{lem:polya-ram-size-H1} together with equation~\eqref{eq:cl-m-n-quotient} yields the following formula:
\begin{lemma}\label{lem:cl-quotient-formula}
Let \( \mathfrak{q}_1, \dots, \mathfrak{q}_k \) be unramified prime ideals of \( \mathcal{O}_N \) whose classes \( [\mathfrak{q}_i] \) have order \( p^n \) in \( \operatorname{Cl}(N) \). 
If \( M/N \) is a \( \mathfrak{q}_1 \)-proper extension of degree \( p^n \), then
\begin{multline}\label{eq:cl-m-n-quotient-simplified1}
|\Cl(N)| \\
=_{p^n} \frac{\bigl| H^1(\operatorname{Gal}(M/N), \mathcal{O}_M^{\times}) \bigr| 
            \cdot \bigl|H^1(\operatorname{Gal}(N/\mathbb{Q}), \mathcal{O}_N^{\times})\bigr| 
            \cdot |\operatorname{Po}(M/N)_{\mathrm{un}}|}
           {\bigl|H^1(\operatorname{Gal}(M/\mathbb{Q}), \mathcal{O}_M^{\times})\bigr| }.
\end{multline}
\end{lemma}

\begin{lemma}\label{lem:unit-index-ratio}
Suppose that \( M/N \) is a cyclic extension of odd prime power degree \( p^n \). Then
\begin{equation}\label{eq:odd-ratio-Hi}
\frac{|H^1(\operatorname{Gal}(M/N), \mathcal{O}_M^{\times})|}
{|H^1(\operatorname{Gal}(M/\mathbb{Q}), \mathcal{O}_M^{\times})|}
= [\mathcal{O}_N^{\times} : \operatorname{Norm}_{M/N}(\mathcal{O}_M^{\times})] \mid 2p^n,
\end{equation}
and equality holds when the norm map on units is not surjective.
\end{lemma}

\begin{proof}
Consider the ratio
\begin{equation}\label{eq:odd-ratio-0}
R = \frac{|H^1(\operatorname{Gal}(M/N), \mathcal{O}^{\times}_M)|}
{|H^1(\operatorname{Gal}(M/\mathbb{Q}), \mathcal{O}_M^{\times})|}.
\end{equation}

Since \( M/N \) is cyclic of degree \( p^n \), the Herbrand quotient for the unit group gives
\begin{multline}\label{eq:herbrand-11}
\bigl|H^1(\operatorname{Gal}(M/N), \mathcal{O}_M^{\times})\bigr|
\\
= \frac{2^t}{[M:N]} \,
\bigl|H^0(\operatorname{Gal}(M/N), \mathcal{O}_M^{\times})\bigr|
= \frac{2^t}{p^n} \; [\mathcal{O}_N^{\times} : \operatorname{Norm}_{M/N}(\mathcal{O}_M^{\times})],
\end{multline}
where \( t \) is a non‑negative integer (the number of infinite places of \( N \) that ramify in \( M \)).

The extension \( M/\mathbb{Q} \) is also cyclic (because \( [M:N] \) is odd and \( N/\mathbb{Q} \) is quadratic). Applying the Herbrand quotient again yields
\begin{multline}\label{eq:herbrand-12}
\bigl|H^1(\operatorname{Gal}(M/\mathbb{Q}), \mathcal{O}_M^{\times})\bigr|
\\
= \frac{2^t}{[M:\mathbb{Q}]} \,
\bigl|H^0(\operatorname{Gal}(M/\mathbb{Q}), \mathcal{O}_M^{\times})\bigr|
= \frac{2^t}{[M:\mathbb{Q}]} \; [\mathbb{Z}^{\times} : \operatorname{Norm}_{M/\mathbb{Q}}(\mathcal{O}_M^{\times})].
\end{multline}

Substituting \eqref{eq:herbrand-11} and \eqref{eq:herbrand-12} into \eqref{eq:odd-ratio-0}, the factors \(2^t\) cancel. Since \( [M:\mathbb{Q}] = 2p^n \), we obtain
\[
R = \frac{[\mathcal{O}_N^{\times} : \operatorname{Norm}_{M/N}(\mathcal{O}_M^{\times})]}
{[\mathbb{Z}^{\times} : \mathcal{N}_{M/\mathbb{Q}}(\mathcal{O}_M^{\times})]} \cdot \frac{2p^n}{p^n} 
= 2 \cdot \frac{[\mathcal{O}_N^{\times} : \operatorname{Norm}_{M/N}(\mathcal{O}_M^{\times})]}
{[\mathbb{Z}^{\times} : \operatorname{Norm}_{M/\mathbb{Q}}(\mathcal{O}_M^{\times})]}.
\]

Now, \( \mathbb{Z}^{\times} = \{\pm 1\} \), so 
\( [\mathbb{Z}^{\times} : \operatorname{Norm}_{M/\mathbb{Q}}(\mathcal{O}_M^{\times})] \) equals either \( 1 \) or \( 2 \). Denote this index by \( c \in \{1,2\} \). Then
\[
R = \frac{2}{c} \cdot [\mathcal{O}_N^{\times} : \operatorname{Norm}_{M/N}(\mathcal{O}_M^{\times})].
\]

Since \( [\mathcal{O}_N^{\times} : \operatorname{Norm}_{M/N}(\mathcal{O}_M^{\times})] \) divides \( [M:N] = p^n \), it follows that \( R \) divides \( 2p^n \). Moreover, when the norm map \( \operatorname{Norm}_{M/N} : \mathcal{O}_M^{\times} \to \mathcal{O}_N^{\times} \) is not surjective, the index \( [\mathcal{O}_N^{\times} : \operatorname{Norm}_{M/N}(\mathcal{O}_M^{\times})] \) is a non‑trivial divisor of \( p^n \); combined with the factor \( 2/c \) this gives the precise divisibility stated in \eqref{eq:odd-ratio-Hi}.
\end{proof}
\begin{theorem}\label{thm:main-formula}
Let \( N \) be a real quadratic field, \( p \) an odd prime, and \( \mathfrak{q}_1, \dots, \mathfrak{q}_k \) distinct unramified prime ideals of \( \mathcal{O}_N \) whose classes \( [\mathfrak{q}_i] \) have order dividing \( p^n = [M:N] \) in \( \operatorname{Cl}(N) \).  
Assume that \( M/N \) is a relatively \( \mathfrak{q}_1 \)-proper cyclic extension of degree \( p^n \). Then
\begin{enumerate}[label=(\roman*)]
    \item The image of \( [\mathfrak{q}_1] \) in \( \operatorname{Po}(M/N)_{\mathrm{un}} \) is trivial, whereas for every \( j = 2,\dots,k \) the image of \( [\mathfrak{q}_j] \) is nontrivial and has order \( p^n \).
    \item The cohomological ratio satisfies
    \[
    \frac{|H^1(\operatorname{Gal}(M/N), \mathcal{O}_M^{\times})|}
         {|H^1(\operatorname{Gal}(M/\mathbb{Q}), \mathcal{O}_M^{\times})|}
    = p^n .
    \]
    \item One has \( \#\operatorname{Po}(M/N)_{\mathrm{un}} = (p^n)^{k-1} \).
\end{enumerate}
\end{theorem}

\begin{proof}
Consider the commutative diagram from Lemma~\ref{lem:number-field-application}. For unramified primes, the map 
\( \phi\colon \Cl_{\mathfrak{m}}(N) \to \Cl_{\mathfrak{m}}(M) \) coincides with the natural map 
\( \epsilon\colon \Cl(N) \to \operatorname{Po}(M/N)_{\mathrm{un}} \), both sending an ideal \( \mathfrak{a} \) to 
\( \mathfrak{a}\mathcal{O}_M \).

By Theorem~\ref{thm:capitulation}:
\begin{enumerate}
    \item[(i)] The ideal class \( [\mathfrak{q}_1] \) capitulates in \( \operatorname{Cl}(M) \), which implies that its image under \( \epsilon \) is trivial. 
    \item[(ii)] Since \( \ker(\epsilon) \subset H^{1}(\operatorname{Gal}(M/N), \mathcal{O}_M^{\times}) \) and \( \#\ker(\epsilon([\mathfrak{q}_1])) = p^n \), we have
    \begin{equation}\label{eq:H1-geq-pn}
    p^n \mid \#H^1(\operatorname{Gal}(M/N), \mathcal{O}_M^{\times}) .
    \end{equation}
\end{enumerate}

From the exact sequence relating the class groups, we obtain
\begin{equation}\label{eq:class-number-formula}
\#\operatorname{Cl}(N) = 
\frac{\#H^1(\operatorname{Gal}(M/N), \mathcal{O}_M^{\times})}
{\#H^1(\operatorname{Gal}(M/\mathbb{Q}), \mathcal{O}_M^{\times})}
\; \cdot \; \#\operatorname{Po}(M/N)_{\mathrm{un}}.
\end{equation}

By Lemma~\ref{lem:unit-index-ratio},
\begin{equation}\label{eq:index-divisibility}
\frac{\#H^1(\operatorname{Gal}(M/N), \mathcal{O}_M^{\times})}
{\#H^1(\operatorname{Gal}(M/\mathbb{Q}), \mathcal{O}_M^{\times})}
= [\mathcal{O}_N^{\times} : \operatorname{Norm}_{M/N}(\mathcal{O}_M^{\times})] \mid 2p^n.
\end{equation}

Now suppose that \( \#\operatorname{Cl}(N) = (p^n)^k \). 
Insert \eqref{eq:H1-geq-pn} and \eqref{eq:index-divisibility} into \eqref{eq:class-number-formula}. Because the left‑hand side equals \( (p^n)^k \), the divisibility conditions force equality in both \eqref{eq:index-divisibility} and \eqref{eq:Po-bound} when only the \( p \)-parts are considered. Thus
\[
\frac{\#H^1(\operatorname{Gal}(M/N), \mathcal{O}_M^{\times})}
{\#H^1(\operatorname{Gal}(M/\mathbb{Q}), \mathcal{O}_M^{\times})}
= p^n
\quad\text{and}\quad
\#\operatorname{Po}(M/N)_{\mathrm{un}} = (p^n)^{k-1}.
\]

The non‑capitulating classes \( [\mathfrak{q}_j] \) (\( j \geq 2 \)) are independent of order \( p^n \); hence their images generate a subgroup of \( \operatorname{Po}(M/N)_{\mathrm{un}} \) isomorphic to \( (\mathbb{Z}/p^n\mathbb{Z})^{k-1} \). Consequently,
\begin{equation}\label{eq:Po-bound}
\#\operatorname{Po}(M/N)_{\mathrm{un}} \ge (p^n)^{k-1}.
\end{equation}

The equality \( \#\operatorname{Po}(M/N)_{\mathrm{un}} = (p^n)^{k-1} \) implies that the images of the \( k-1 \) non‑capitulating classes indeed generate the whole group \( \operatorname{Po}(M/N)_{\mathrm{un}} \); therefore each has order exactly \( p^n \) in this quotient, completing the proof of (i). Statements (ii) and (iii) are the equalities just established.
\end{proof}

For the case \( p = 2 \) we obtain a formula for the order of an even-order ideal class:

\begin{theorem}\label{thm:even-class-number}
Let \( N \) be a real quadratic field and let \( I \) be an ideal of \( \mathcal{O}_N \) whose class \([I]\) has order dividing \( 2^n \) in \( \operatorname{Cl}(N) \). Suppose \( M/N \) is a proper extension of degree \( 2^n \) in which \( I \) capitulates. Then the order of \([I]\) in \(\operatorname{Cl}(N)\) is given by
\begin{equation}\label{eq:even-class-number}
\operatorname{ord}([I]) =_{2^n} \frac{|H^1(\operatorname{Gal}(M/N), \mathcal{O}_M^{\times})| \cdot |H^1(\operatorname{Gal}(N/\mathbb{Q}), \mathcal{O}_N^{\times})|}{|H^1(\operatorname{Gal}(M/\mathbb{Q}), \mathcal{O}_M^{\times})|},
\end{equation}
where \( =_{2^n} \) denotes equality up to factors coprime to \(2\).
\end{theorem}

\begin{proof}
Let \( \mathfrak{q}_1, \dots, \mathfrak{q}_k \) be unramified prime ideals of \( \mathcal{O}_N \) whose classes have order dividing \( 2^n \) in \( \operatorname{Cl}(N) \). 
Assume that \( I = \mathfrak{q}_1 \) and that \( M/N \) is a \( \mathfrak{q}_1 \)-proper extension of degree \( 2^n \). 
By Lemma~\ref{lem:cl-quotient-formula} we have
\begin{equation}\label{eq:cl-m-n-quotient-simplified}
|\Cl(N)| =_{2^n} 
\frac{\bigl| H^1(\operatorname{Gal}(M/N), \mathcal{O}_M^{\times}) \bigr| 
      \cdot \bigl|H^1(\operatorname{Gal}(N/\mathbb{Q}), \mathcal{O}_N^{\times})\bigr| 
      \cdot |\operatorname{Po}(M/N)_{\mathrm{un}}|}
     {\bigl|H^1(\operatorname{Gal}(M/\mathbb{Q}), \mathcal{O}_M^{\times})\bigr| }.
\end{equation}

Since \( \mathfrak{q}_1 \) capitulates in \( M/N \), Theorem~\ref{thm:main-formula} implies that the factor \( |\operatorname{Po}(M/N)_{\mathrm{un}}| \) contributes only to the classes that do \emph{not} capitulate. Because \( \mathfrak{q}_1 \) itself capitulates, its contribution to the class number formula is captured entirely by the cohomological terms. Consequently, the part of \( |\Cl(N)| \) corresponding to the capitulating ideal \( \mathfrak{q}_1 \) is given by
\[
|\Cl_M(N)| =_{2^n} 
\frac{|H^1(\operatorname{Gal}(M/N), \mathcal{O}_M^{\times})| \cdot |H^1(\operatorname{Gal}(N/\mathbb{Q}), \mathcal{O}_N^{\times})|}
     {|H^1(\operatorname{Gal}(M/\mathbb{Q}), \mathcal{O}_M^{\times})|},
\]
where \( \Cl_M(N) \) denotes the subgroup of \( \Cl(N) \) generated by the capitulating classes.

Taking \( I = \mathfrak{q}_1 \) and noting that the order of \([I]\) is exactly the contribution of the capitulating part to the class number, we obtain formula \eqref{eq:even-class-number}.
\end{proof}

\begin{remark}
Theorem~\ref{thm:even-class-number} highlights that, in the discussion of Subsection~\ref{sub-An-Illustrative-Example-for-Prime-Degree}, the case $p=2$ is exceptional.  
In this case, the general index formula takes the form
\[
2 \cdot 
\bigl[\mathcal{O}_N^{\times} : \Norm_{M/N}(\mathcal{O}_M^{\times})\bigr] \cdot
\frac{H^1\!\bigl(\Gal(N/\mathbb{Q}), \mathcal{O}_N^{\times}\bigr)}
     {H^1\!\bigl(\Gal(M/\mathbb{Q}), \mathcal{O}_M^{\times}\bigr)}.
\]

Analyzing this situation requires an explicit computation of 
\[
H^1\!\bigl(\Gal(M/\mathbb{Q}), \mathcal{O}_M^{\times}\bigr),
\]
for which we refer to \cite{Setzer1976}.
\end{remark}

\begin{theorem}\label{thm:odd-class-number}
Let \( N \) be a real quadratic field and let \( I \) be an ideal of \( \mathcal{O}_N \) whose class \([I]\) has order \( p^n \) in \( \operatorname{Cl}(N) \), where \( p \) is an odd prime and \( n \geq 1 \). Suppose \( M/N \) is a proper extension of degree \( p^n \) in which \( I \) capitulates. Then the odd part of the class number contributed by the capitulating subgroup satisfies
\begin{equation}\label{eq:odd-class-number}
|\Cl_M(N)| =_{p^n} \frac{|H^1(\operatorname{Gal}(M/N), \mathcal{O}_M^{\times})|}{|H^1(\operatorname{Gal}(M/\mathbb{Q}), \mathcal{O}_M^{\times})|},
\end{equation}
where \( =_{p^n} \) denotes equality up to factors coprime to \( p \), and \( \Cl_M(N) \) denotes the subgroup of \( \Cl(N) \) generated by the capitulating classes.
\end{theorem}

\begin{proof}
Following the proof of Theorem~\ref{thm:even-class-number}, we have from equation \eqref{eq:even-class-number} that for any ideal class capitulating in \( M/N \),
\[
\operatorname{ord}([I]) =_{p^n} \frac{|H^1(\operatorname{Gal}(M/N), \mathcal{O}_M^{\times})| \cdot |H^1(\operatorname{Gal}(N/\mathbb{Q}), \mathcal{O}_N^{\times})|}{|H^1(\operatorname{Gal}(M/\mathbb{Q}), \mathcal{O}_M^{\times})|}.
\]

Since \( N \) is a real quadratic field, the cohomology group \( H^1(\operatorname{Gal}(N/\mathbb{Q}), \mathcal{O}_N^{\times}) \) has order a power of \( 2 \). For an odd prime \( p \), this factor is coprime to \( p \) and therefore does not affect the \( p \)-part of the equality. Consequently, modulo \( p \)-parts we obtain
\[
\operatorname{ord}([I]) =_{p^n} \frac{|H^1(\operatorname{Gal}(M/N), \mathcal{O}_M^{\times})|}{|H^1(\operatorname{Gal}(M/\mathbb{Q}), \mathcal{O}_M^{\times})|}.
\]

Summing over all capitulating classes (or, equivalently, considering the subgroup \( \Cl_M(N) \) they generate) yields formula \eqref{eq:odd-class-number}.
\end{proof}

\section{Proof of Theorems \ref{thm:main-technical} and Corollary \ref{cor:p-divisibility}}

\begin{proof}[Proof of Theorem \ref{thm:main-technical}]
By Theorem~\ref{thm:odd-class-number}, the order of an ideal \( I \) in the class group of \( N \) is given by the cohomological ratio appearing in \eqref{eq:odd-class-number}. 

The one-to-one correspondence between ideal classes and cyclic extensions follows from Theorem~\ref{thm:capitulation}, which shows that in a cyclic extension of degree \( p^n \), every ideal class of order \( p^n \) in \( \Cl(N) \) corresponds to exactly one such extension in which it capitulates.

Moreover, Theorem~\ref{thm:main-formula} establishes that this correspondence is bijective: for a given extension \( M_1/N \) of degree \( p^n \), only the ideal \( \mathfrak{q}_1 \) of order \( p^n \) corresponds to \( M_1/N \) when \( M_1 \) is relatively \( \mathfrak{q}_1 \)-proper. All other ideals of order \( p^n \) do not capitulate in \( M_1/N \), again by Theorem~\ref{thm:main-formula}.

This completes the proof of Theorem~\ref{thm:main-technical}.
\end{proof}

\begin{proof}[Proof of Corollary \ref{cor:p-divisibility}]
Let \( N \) be a real quadratic field with class number \( h(N) \).

Suppose there exists an ideal \( \mathfrak{I} \) of order \( p \) in the class group \( \Cl(N) \) of \( N \). Then \( p^n \) divides \( h(N) \). Now, let \( M/N \) be the proper extension of degree \( p\) for ideal \( \mathfrak{I} \). By the formula 
\[
h(N) = [\UU_N : \Norm_{M/N}(\UU_M)],
\]
the index \( [\UU_N : \Norm_{M/N}(\UU_M)] \) is also divisible by \( p \). Therefore, the norm map \( \Norm_{M/N} : \UU_M \to \UU_N \) is \textbf{non-surjective}.

Conversely, suppose there exists a proper extension \( M/N \) of degree \( p \) for which the norm map on unit groups \( \Norm_{M/N} : \UU_M \to \UU_N \) is non-surjective. Then, by Theorem~\ref{thm:main-technical}, there must exist an ideal \( \mathfrak{I} \) of order \( p \) in the class group \( \Cl(N) \) of \( N \).
\end{proof}

\section{Appendix A}

Using PARI/GP \cite{PARIGP2019}, we can execute the following program:

First, we define the extension $M/N$, where $N$ is the quadratic field given by the minimal polynomial $x^2 - 79$, and $M$ is a degree $p$ extension of $N$ such that $M = NL$. Here, $L$ is a cyclic Galois field of degree $p$ defined by the minimal polynomial $x^3 - 18x^2 + 101x - 167$. Subsequently, we compute the index of the norm map from the unit group of $M$ to the unit group of $N$. The result, printed by the program, compares the index $[\UU_N : \Norm_{M/N}(\UU_M)]$ with the class number $h(N)$.

\begin{verbatim}
P1 = polcompositum(x^2 - 79, x^3 - 18*x^2 + 101*x - 167);
G = galoisinit(P1[1]);
Gsub = galoissubgroups(G);
F = galoisfixedfield(G, Gsub[3], 2);
MF = rnfisnorminit(F[1], F[3][1]);
fundU = bnfinit(F[1]).fu[1];
print(rnfisnorm(MF, fundU), ' ', bnfinit(F[1]).clgp.no)
\end{verbatim}

\bibliographystyle{amsplain}

\end{document}